\PassOptionsToPackage{backref=page}{hyperref}
\documentclass[12pt,letterpaper,reqno]{amsart}

\usepackage[T1]{fontenc}
\usepackage[utf8]{inputenc}
\usepackage{amsmath,amssymb,amsfonts,amsthm}
\usepackage{mathtools}
\usepackage{aliascnt}
\usepackage{microtype}
\usepackage{enumitem}
\usepackage{booktabs}
\usepackage{xcolor}
\usepackage{doi}
\usepackage{hyperref}
\usepackage{tikz}
\usetikzlibrary{positioning,arrows.meta}
\usepackage{bookmark}
\usepackage[capitalize,noabbrev]{cleveref}

\hypersetup{
  pdfstartview={FitH},
  colorlinks=true,
  linkcolor=blue!55!black,
  citecolor=green!35!black,
  urlcolor=blue!55!black
}
\renewcommand*{\backref}[1]{}
\renewcommand*{\backrefalt}[4]{%
  \ifcase #1
  \or
    \quad$\hookleftarrow$, cited on page~#2%
  \else
    \quad$\hookleftarrow$, cited on pages~#2%
  \fi
}

\newtheorem{theorem}{Theorem}[section]

\newaliascnt{lemma}{theorem}
\newtheorem{lemma}[lemma]{Lemma}
\aliascntresetthe{lemma}

\newaliascnt{proposition}{theorem}
\newtheorem{proposition}[proposition]{Proposition}
\aliascntresetthe{proposition}

\newaliascnt{corollary}{theorem}
\newtheorem{corollary}[corollary]{Corollary}
\aliascntresetthe{corollary}

\newaliascnt{conjecture}{theorem}
\newtheorem{conjecture}[conjecture]{Conjecture}
\aliascntresetthe{conjecture}

\newaliascnt{problem}{theorem}

\aliascntresetthe{problem}

\theoremstyle{definition}
\newaliascnt{definition}{theorem}

\aliascntresetthe{definition}

\newaliascnt{example}{theorem}

\aliascntresetthe{example}

\newaliascnt{remark}{theorem}

\aliascntresetthe{remark}

\crefname{theorem}{Theorem}{Theorems}
\Crefname{theorem}{Theorem}{Theorems}
\crefname{lemma}{Lemma}{Lemmas}
\Crefname{lemma}{Lemma}{Lemmas}
\crefname{proposition}{Proposition}{Propositions}
\Crefname{proposition}{Proposition}{Propositions}
\crefname{corollary}{Corollary}{Corollaries}
\Crefname{corollary}{Corollary}{Corollaries}
\crefname{conjecture}{Conjecture}{Conjectures}
\Crefname{conjecture}{Conjecture}{Conjectures}

\DeclareMathOperator{\Aut}{Aut}
\DeclareMathOperator{\End}{End}
\DeclareMathOperator{\Gal}{Gal}
\DeclareMathOperator{\Hom}{Hom}
\DeclareMathOperator{\Irr}{Irr}
\DeclareMathOperator{\rad}{rad}
\DeclareMathOperator{\soc}{soc}
\DeclareMathOperator{\Tr}{Tr}
\newcommand{\modcat}{\operatorname{mod}}

\newcommand{\op}{\mathrm{op}}
\newcommand{\id}{\mathrm{id}}
\newcommand{\kkbar}{\overline{k}}

\begin{document}

\title[Components of Auslander--Reiten quivers]
{Infinitely many components in Auslander--Reiten quivers of representation-infinite algebras over perfect fields}

\author[W.~Chang]{Wen Chang}
\address{School of Mathematics and Statistics, Shaanxi Normal University, Xi'an 710062, P.~R.~China}
\email{changwen161@163.com}

\author[Q.~Tang]{Quanyu Tang}
\address{School of Mathematical Sciences, University of Science and Technology of China, Hefei 230026, P.~R.~China}
\email{tangquanyu827@gmail.com}

\subjclass[2020]{Primary 16G70; Secondary 16G60, 12F10}

\keywords{Auslander--Reiten quiver, connected component, representation-infinite algebra, representation embedding, separable base change}

\begin{abstract}
Let $k$ be a perfect field and let $A$ be a representation-infinite finite-dimensional $k$-algebra. We prove that the Auslander--Reiten quiver of $A$ has infinitely many connected components. This establishes, for finite-dimensional algebras over perfect fields, a conjecture of Auslander, Reiten, and Smal\o{} concerning Artin algebras. Over an algebraically closed field, the proof combines a localized polynomial representation embedding with semilinear twists induced by field automorphisms. The passage from a perfect field to its algebraic closure is obtained by separable base change: we prove that if the Auslander--Reiten quiver of $A$ has only finitely many components, then the same holds for the scalar extension to the algebraic closure. 
\end{abstract}

\maketitle

\section{Introduction}\label{sec:introduction}
Almost split sequences were introduced by Auslander and Reiten
in the 1970s \cite{AR75}, and have since become a cornerstone
of the representation theory of Artin algebras.  They encode
the local morphism-theoretic structure of an algebra by
organizing its indecomposable modules through irreducible
morphisms.  The resulting Auslander--Reiten quiver (AR-quiver for brevity)
assembles this local information into a single global
combinatorial invariant: its vertices are the isomorphism classes of indecomposable modules, while its arrows and valuations record the nonzero spaces of irreducible morphisms between them. Almost split sequences endow the quiver with its translation and mesh structure.

In this way, the AR-quiver
is often the first invariant one computes when approaching a
new class of algebras---its connected components, their
shapes, and the possible cardinalities of components reflect
the representation type and the overall complexity of the
category.  Determining how many components an AR-quiver can
have, and when this number is infinite, is therefore a
fundamental structural question.

An Artin algebra is of infinite representation type if it has infinitely many pairwise nonisomorphic indecomposable finitely generated modules.
For the background on representation theory of Artin algebras, we refer the reader to the monograph of Auslander, Reiten, and
Smal{\o}~\cite{ARS}.  At the end of that book the authors collected
several conjectures and open problems, among them the following
conjecture~\cite[Conjecture~(3), p.~409]{ARS}.

\begin{conjecture}\label{conj:ars}
Let $A$ be an Artin algebra with associated AR-quiver $\Gamma_A$. If $A$ is of infinite representation type, then $\Gamma_A$ has infinitely many connected components.
\end{conjecture}

It has been verified in several classes of algebras, for example, \emph{hereditary Artin algebras} by Dlab and Ringel \cite{DlabRingel, Ringel78}, \emph{finite-dimensional algebras of tame type} over an algebraically closed field by Crawley-Boevey \cite{CrawleyBoevey} (the analogous result over a perfect field, when the scalar extension to the algebraic closure is tame, follows from~\cite[Theorem~3.2 and Corollary~6.16]{CanDeVicentePerez}), a block of a \emph{cocommutative Hopf algebra} whose dual is local by Farnsteiner \cite{Farnsteiner00}, and certain \emph{self-injective algebras} satisfying suitable finite-generation hypotheses on cohomology by K\"ulshammer~\cite{Kuelshammer13,KuelshammerCorrigendum}. 
Our main result settles the problem for any finite-dimensional algebra over a perfect field.

\begin{theorem}\label{thm:main}
Let $k$ be a perfect field, and let $A$ be a finite-dimensional $k$-algebra
of infinite representation type.  Then $\Gamma_A$ has infinitely many
connected components.
\end{theorem}

The proof has two parts.  The first treats the case when the base field $k$ is algebraically
closed.  A representation embedding due to Bongartz \cite{Bongartz} and further developed in \cite{BPS}, from a localization of
a polynomial ring, supplies a one-parameter family of pairwise nonisomorphic
indecomposable $A$-modules of one fixed dimension.  We then let field
automorphisms act on this family through semilinear twists. On the other hand, a uniform local estimate in the Auslander--Reiten quiver then shows that, for every sufficiently large prime $\ell$, the $\ell$-adic valuation of the orbit length is constant on each connected component. By choosing field
automorphisms with successively prescribed prime orbit lengths, we obtain
modules in pairwise distinct components, see \cref{sec:algebraically-closed}.

The argument over an algebraically closed field is only genuinely needed
when the field is countable.  Indeed, the representation embedding produces
$|k|$ pairwise nonisomorphic indecomposable modules of one fixed dimension,
apart from finitely many excluded parameters, cf. \cref{lem:polynomial-embedding} and \cref{eq:specialization-module}.  Since every connected component of
a locally finite graph is countable, an uncountable algebraically
closed field already gives infinitely many components by cardinality.  The
orbit argument mentioned above also covers this case and, more importantly, works over
countable algebraically closed fields such as $\overline{\mathbb Q}$ and
$\overline{\mathbb F}_p$.

Twisting by a field automorphism is a standard semilinear construction:
restriction of scalars gives an additive, generally non-$k$-linear,
equivalence and preserves vector-space dimension.  For the broader role of
Galois actions in the representation theory of Artin algebras and finite
Galois coverings, see~\cite{ARS-Galois};
for a recent systematic use of semilinear twists, see~\cite{BennettCrawleyBoevey,MalicSchroll}.  Our
argument uses the prime divisibility of orbit lengths under such twists to
produce an invariant of connected components of the AR-quiver.

The second part passes from a perfect field $k$ to its algebraic closure
$K=\kkbar$ and lifts the algebra $A$ to the scalar extension $A_K=K\otimes_kA$.  
Jensen and Lenzing proved that finite representation type is
preserved and reflected by MacLane-separable extensions
\cite[Theorem~3.3]{JensenLenzing}; see also \cite{Kasjan} for the behavior of
almost split sequences under such extensions and \cite{BHK25} for more recent applications on $\tau$-tilting theory.  We establish the particular
component-theoretic statement needed here, that is, if $\Gamma_A$ has only finitely
many components, then $\Gamma_{A_K}$ has only finitely many components.
This is not trivial, since scalar extension can split an indecomposable
module.  The proof uses base-change formulas for the radical and
its square, transitivity of $\Gal(K/k)$ on the indecomposable summands of an
extended indecomposable module, and a path-lifting argument starting at a
prescribed summand, see \cref{sec:base-change}.

We mention that the perfectness of $k$ ensures that $K/k$ is
separable and that finite-dimensional semisimple $k$-algebras remain
semisimple after scalar extension.  Without separability, nilpotents can
appear in the scalar extension of a semisimple quotient, and the radical
formulas established in \cref{subsec:radical-base-change}
 need not hold.

The paper is organized as follows.  In \cref{sec:preliminaries} we collect
the field-theoretic and representation-theoretic facts used later, and
record the representation-embedding theorem in the precise form needed.
\Cref{sec:local} establishes the local dimension estimate and the orbit-length
invariant.  The algebraically closed case is proved in
\cref{sec:algebraically-closed}.  Finally, \cref{sec:base-change} develops
the base-change and path-lifting arguments and proves \cref{thm:main}.

\section{Preliminaries}\label{sec:preliminaries}

This section fixes notation and records the standard facts used in the
proof.  The discussion is included partly to make clear where separability
is required and partly to make the conventions used in the later arguments explicit.

\subsection{Field extensions}\label{subsec:field-preliminaries}
We collect in this subsection the field-theoretic facts, including properties of separable and Galois extensions,
cyclotomic subfields, and the extension of automorphisms from a finite
Galois subextension to an algebraically closed overfield.

\subsubsection{Perfect fields and Galois extensions}\label{Perfect fields and Galois extensions}
Let $k$ be a field and let $\kkbar$ be an algebraic closure.  An algebraic
extension $K/k$ is \emph{separable} if the minimal polynomial over $k$ of
every element of $K$ has no repeated root.  The field $k$ is
\emph{perfect} if every algebraic extension of $k$ is separable.  Fields of
characteristic zero and finite fields are perfect.  In characteristic
$p>0$, perfectness is equivalent to surjectivity of the Frobenius map
$x\mapsto x^p$.

An algebraic extension $K/k$ is \emph{normal} if every irreducible polynomial
over $k$ that has a root in $K$ splits completely in $K$; equivalently, if every $k$-embedding of $K$ into $\kkbar$ maps $K$ to itself.  
 We say that $K/k$ is a \emph{Galois
extension} if it is separable and normal; equivalently, $k$ is the fixed field of $k$-automorphisms $\Aut_k(K)$ of $K$. In this case
we write $G=\Gal(K/k)$ for the automorphism group $\Aut_k(K)$, called the
\emph{Galois group} of $K/k$.
In particular, if $k$ is perfect and $K=\kkbar$, then $K/k$ is a
Galois extension and $K^{\Gal(K/k)}=k$.

The theorem of Jensen and Lenzing that will be used at the
end of the paper states that, for a MacLane-separable extension $K/k$ and a
finite-dimensional $k$-algebra $A$,
\[
A\text{ is of finite representation type}
\quad\Longleftrightarrow\quad
K\otimes_k A\text{ is of finite representation type};
\]
see \cite[Theorem~3.3]{JensenLenzing}.
Note that 
for algebraic field extensions, MacLane separability is equivalent to ordinary separability.

\subsubsection{Linear disjointness and regular extensions}\label{Linear disjointness and regular extensions}
Let $K/k$ and $L/k$ be subextensions of a common overfield.  They are
\emph{linearly disjoint over $k$} if the multiplication map
$$
K\otimes_k L\longrightarrow KL,
\qquad x\otimes y\longmapsto xy,
$$
is injective. Equivalently, every finite subset of $K$ that is linearly
independent over $k$ remains linearly independent over $L$ in $KL$; the
analogous statement holds with $K$ and $L$ interchanged.
If $L/k$ is finite Galois and $K$ and $L$ are linearly disjoint over $k$,
then $KL/K$ is Galois and restriction gives an isomorphism
\begin{equation}\label{eq:galois}
\Gal(KL/K)\simeq \Gal(L/k).
\end{equation}

A field extension $K/k$ is called \emph{regular} if it is
separable and $k$ is relatively algebraically closed in $K$, that is, every element of $K$ algebraic over
$k$ already lies in $k$.  A regular extension is
linearly disjoint from every algebraic extension of the ground field.  We
will use the following standard consequence.  Let $K$ be finitely generated
over its prime field $F$, and let $F^{alg}_K$ be the relative algebraic closure of $F$ in $K$, that is,
$$
F^{alg}_K=\{x\in K\mid x\text{ is algebraic over }F\}.
$$
Then $K/F^{alg}_K$ is regular.  In characteristic $p>0$,
$F^{alg}_K$ is a finite field; in characteristic zero, $F^{alg}_K$ is a number field.
These facts may be found in
\cite[Chapters~V and VI]{Lang}.

\subsubsection{Prime-degree cyclic extensions}\label{Prime-degree cyclic extensions}
If $K/k$ is cyclic Galois of prime degree $\ell$ and $\rho$ is a generator
of $\Gal(K/k)$, then every $\lambda\in K\setminus k$ has a $\rho$-orbit of
length exactly $\ell$.  In positive characteristic we shall obtain such an
extension by adjoining a finite field extension.  In characteristic zero we
use the following cyclotomic fact.  For an odd prime $\ell$,
$$
\Gal\bigl(\mathbb Q(\zeta_{\ell^2})/\mathbb Q\bigr)
\simeq (\mathbb Z/\ell^2\mathbb Z)^\times
$$
is cyclic of order $\ell(\ell-1)$.  Hence
$\mathbb Q(\zeta_{\ell^2})$ contains a cyclic subextension of degree
$\ell$ over $\mathbb Q$; see \cite[Chapter~2]{Washington}.

\subsubsection{Extension of automorphisms}\label{Extension of automorphisms}
Let $F\subseteq k$ with $k$ algebraically closed, and let
$F^\mathrm{alg}_k$ be the algebraic closure of $F$ contained in $k$.  If $E/F$ is a finite
Galois subextension of $F^\mathrm{alg}_k/F$, every element of
$\Gal(E/F)$ extends to an element of $\Aut_F(k)$.  Indeed, first extend the
automorphism to $F^\mathrm{alg}_k$, choose a transcendence basis
$\mathcal T$ of $k/F^\mathrm{alg}_k$ and fix it pointwise, and then extend
to the algebraic closure $k$ of $F^\mathrm{alg}_k(\mathcal T)$.  We shall
use this construction in \cref{lem:prime-orbits}; the extended automorphism
has the same orbit on every element of $E$ as the original automorphism.

\subsection{Modules and Auslander--Reiten quivers}
\label{subsec:representation-preliminaries}

Let $A$ be a finite-dimensional algebra over a field $k$, and denote by
$\modcat A$ the category of finite-dimensional left $A$-modules, which is a Hom-finite Krull--Schmidt category: every object is a finite
direct sum of indecomposable objects, uniquely up to permutation and
isomorphism, and the endomorphism algebra of an indecomposable object is
local.

\subsubsection{The categorical radical}
For $U,V\in\modcat A$, the \emph{radical}
$\rad_A(U,V)$ consists of the morphisms $f:U\to V$ such that
$1_U-gf$ is invertible for every $g:V\to U$.  Its square is
\[
\rad_A^2(U,V)
 =\sum_{W\in\modcat A}
 \rad_A(W,V)\,\rad_A(U,W),
\]
where the sum means the subspace spanned by all such compositions.  We set
\[
\Irr_A(U,V)=\rad_A(U,V)/\rad_A^2(U,V).
\]
If $U$ and $V$ are indecomposable, a morphism $U\to V$ is \emph{irreducible} if
and only if it belongs to $\rad_A(U,V)$ but not to
$\rad_A^2(U,V)$.

For later use, if $E=\End_A(U\oplus V)$ and $e_U,e_V\in E$ are the
idempotents associated with the two summands, then
\begin{equation}\label{eq:radical-block}
\rad_A(U,V)=e_VJ(E)e_U,
\end{equation}
where $J(E)=\rad_A(U\oplus V,U\oplus V)$ is the \emph{Jacobson radical} of $E$.
This is the usual block description of the radical in a Krull--Schmidt
category; see \cite[Chapter~I, Section~2]{ARS}.

\subsubsection{Auslander--Reiten quiver}
The Auslander--Reiten quiver $\Gamma_A$ is the valued quiver whose vertices are the isomorphism classes of indecomposable $A$-modules. For indecomposable modules $U$ and $V$, there is an arrow from $[U]$ to $[V]$ precisely when $\Irr_A(U,V)\neq 0$. The valuation records the dimensions of $\Irr_A(U,V)$ over the corresponding residue division algebras, namely
$$
  D_V = \End_A(V)/\rad\End_A(V)
  \quad\text{and}\quad
  D_U = \End_A(U)/\rad\End_A(U).
$$

Connectedness of $\Gamma_A$ always refers to the
underlying unoriented graph, that is, the graph obtained by ignoring the orientations and the valuations of the arrows in $\Gamma_A$. 
Since only connectedness is relevant below, we suppress valuations and arrow multiplicities and retain only the existence of an arrow.
We write $U\longrightarrow V$
to mean that $\Gamma_A$ has an arrow from $[U]$ to $[V]$, and in sums indexed
by arrows, each isomorphism class is counted once, irrespective of the
valuation of arrows, see for example \cref{eq:local-estimate}.

\subsubsection{Projective covers and translations} 

We denote by $P(X)\twoheadrightarrow X$ the projective cover for a module $X\in\modcat A$.  A minimal projective
presentation of $X$ is an exact sequence of the form
\[
P_1\longrightarrow P_0\longrightarrow X\longrightarrow0,\]
where $P_0=P(X)$ and $P_1=P(\Omega X)$.
Applying $\Hom_A(-,A)$ gives the transpose
\begin{equation}\label{eq:tr}
\Tr_AX
 =\operatorname{coker}\bigl(
 \Hom_A(P_0,A)\longrightarrow\Hom_A(P_1,A)
 \bigr)
\end{equation}
which is a right $A$-module. 
If $X$ is
indecomposable and nonprojective, its \emph{Auslander--Reiten translation} is $\tau_AX=D\Tr_AX$, where $D=\Hom_k(-,k)$ is the usual $k$-duality between $\modcat A$ and $\modcat A^{\op}$.  
In this case, there is an \emph{almost split sequence}
\[
0\longrightarrow\tau_AX\longrightarrow E^-(X)
\longrightarrow X\longrightarrow0
\]
ending at $X$, and every immediate predecessor of $X$ in the AR-quiver $\Gamma_A$ occurs as an indecomposable direct
summand of $E^-(X)$.

Dually, $\tau_A^{-1}X$ is defined for a noninjective indecomposable $X$, and
\begin{equation}\label{eq:inverse-translate-duality}
D(\tau_A^{-1}X)\simeq \tau_{A^{\op}}(DX),
\end{equation}
and there is an almost split
sequence starting at $X$,
\[
0\longrightarrow X\longrightarrow E^+(X)
\longrightarrow\tau_A^{-1}X\longrightarrow0,
\]
and every immediate successor of $X$ in $\Gamma_A$ occurs as an indecomposable direct
summand of $E^+(X)$.  

If $X$ is indecomposable projective, the inclusion
$\rad X\hookrightarrow X$ is \emph{minimal right almost split}, while
if $X$ is indecomposable injective, the quotient
$X\twoheadrightarrow X/\soc X$ is \emph{minimal left almost split}.
These statements are standard; for details, we refer to
\cite[Chapters~IV and V]{ARS}.

\subsubsection{Additive equivalences and irreducible morphisms}\label{subsubse:add-equ}

An \emph{additive autoequivalence} of $\modcat A$ is an autoequivalence
$\Phi$ that is additive on morphisms, or equivalently, preserves finite
direct sums.  Such an equivalence automatically preserves indecomposable
objects, split morphisms, and factorizations, and therefore preserves
irreducible morphisms.  It follows that $\Phi$ acts on the underlying graph
of $\Gamma_A$. Note that we do not assume that $\Phi$ is $k$-linear.

\subsection{Representation embeddings}\label{subsec:representation-embeddings}
The representation embedding used below originates in work on the Brauer--Thrall conjectures. For historical background, we refer the reader to \cite{Roiter,Ringel,Bautista85,Bongartz85}. More recently, Bongartz strengthened the method to a tensor
representation embedding from the category of $k[T]$-modules
\cite[Theorem~6 (i)]{Bongartz}.
We rely on a version of Bautista, P\'erez, and Salmer\'on \cite{BPS},
which yields a representation embedding from a localized polynomial ring
$k[T,h(T)^{-1}]$ for a nonzero polynomial $h(T)$.

For $k$-algebras $R$ and $A$, a \emph{representation embedding} \(\mathcal F:\modcat R\longrightarrow\modcat A\) is an exact $k$-linear functor which sends indecomposable modules to indecomposable modules and reflects isomorphism classes. The last condition means that, for $R$-modules $X$ and $Y$, $\mathcal F(X)\simeq \mathcal F(Y)$ implies $X\simeq Y$.

The representation embeddings used here are tensor functors.  If
${}_AM_R$ is an $A$--$R$-bimodule and $M_R$ is finitely generated
projective, then
\[
M\otimes_R-:\modcat R\longrightarrow\modcat A
\]
is exact and sends finite-dimensional modules to finite-dimensional
modules.  
We will use the following result, which is proved in \cite{BPS} for a basic algebra $A$, with the right bimodule $M$ initially known to be finitely generated projective. For the convenience of the readers, we include a proof here. In particular, we explain how the cited result yields the statement
below for an arbitrary finite-dimensional algebra $A$, not necessarily
basic, and why the resulting bimodule is free of finite rank as a right
$R$-module.

\begin{lemma}\label{lem:polynomial-embedding}
Let $k$ be algebraically closed, and let $A$ be a representation-infinite
finite-dimensional $k$-algebra. Then there exist a nonzero polynomial
$h(T)\in k[T]$, the localization $R=k[T,h(T)^{-1}]$, and an $A$--$R$-bimodule $M$, free of finite rank as a right $R$-module,
such that
\[
H=M\otimes_R-:\modcat R\longrightarrow\modcat A
\]
is a representation embedding.
\end{lemma}

\begin{proof} Choose a basic algebra $B$ Morita equivalent to $A$.  Since $k$ is algebraically closed, $B/J(B)$ is a finite product of copies of $k$. By \cite[Remark~4.10]{BPS}, there exist a nonzero polynomial
$h(T)$, a $B$--$R$-bimodule $Z$ which is finitely generated projective as a
right $R=k[T,h(T)^{-1}]$-module, and a representation embedding
\[
Z\otimes_R-:\modcat R\longrightarrow\modcat B.
\]

Let ${}_AP_B$ be a Morita bimodule such that
$P\otimes_B-:\modcat B\to\modcat A$ is an equivalence, and set
$M=P\otimes_BZ$.  Since $P_B$ is finitely generated projective, it is a
direct summand of $B^s$ for some $s$, that is, there is a right $B$-module $Q$ such that $P\oplus Q\simeq B^s$.
Tensoring this isomorphism with $Z$ over $B$ gives
\[
(P\otimes_BZ)\oplus(Q\otimes_BZ)\simeq Z^s
\]
as right $R$-modules.
Hence $M$ is a finitely
generated projective right $R$-module.  The ring $R$ is a localization of the principal
ideal domain $k[T]$, so it is again a principal ideal domain.  Therefore
$M_R$ is free of finite rank.  Finally,
\[
M\otimes_R-
 \simeq P\otimes_B\bigl(Z\otimes_R-\bigr),
\]
so the tensor functor defined by $M$ is the composite of a representation
embedding with a Morita equivalence, which is still a representation embedding.
\end{proof}

For $\lambda\in k$ with $h(\lambda)\neq0$, let
\begin{equation}\label{eq:specialization-module}
S_\lambda=R/(T-\lambda)R.
\end{equation}
This is a one-dimensional $R$-module on which $T$ acts as multiplication by
$\lambda$, and
\[
S_\lambda\simeq S_\mu\quad\Longleftrightarrow\quad\lambda=\mu.
\]
If $M_R$ has rank $t$, then $H(S_\lambda)$ has $k$-dimension $t$.
Consequently, \cref{lem:polynomial-embedding} gives a one-parameter family
of pairwise nonisomorphic indecomposable $A$-modules of one fixed dimension,
with the roots of $h$ excluded.

\section{Local estimates and orbit lengths}\label{sec:local}

Let $k$ be a field and let $A$ be a finite-dimensional $k$-algebra.  Put
\begin{equation*}
d=\dim_kA,
\qquad
C_A=1+d^2.
\end{equation*}
Recall that if $X$ and $Y$ are indecomposable $A$-modules, we write $X\longrightarrow Y$ when
$\Gamma_A$ has an arrow from $[X]$ to $[Y]$.  

\subsection{A uniform local estimate}\label{subsec:local-estimate}

\begin{lemma}\label{lem:local-estimate}
For every indecomposable $A$-module $X$,
\begin{equation}\label{eq:local-estimate}
\sum_{X\longrightarrow Y}\dim_kY\leq C_A\dim_kX,
\qquad
\sum_{Y\longrightarrow X}\dim_kY\leq C_A\dim_kX,
\end{equation}
where in the sums each indecomposable isomorphism class of $Y$ is counted once.
\end{lemma}

\begin{proof}
Let $P(X)\twoheadrightarrow X$ be a projective cover of $X$.  A $k$-basis of $X$ gives an epimorphism
$A^{\dim_kX}\twoheadrightarrow X$.  Since a projective cover is a direct
summand of every projective module admitting an epimorphism onto the same
module, we have $\dim_kP(X)\leq d\dim_kX$.

Choose a minimal projective presentation $P_1\longrightarrow P_0\longrightarrow X\longrightarrow0$, where $P_0=P(X)$ and $P_1=P(\Omega X)$.  Hence
\[
\dim_k\Omega X\leq\dim_kP_0\leq d\dim_kX.
\]

Furthermore, since $P_1$ is a direct
summand of $A^{\dim_k\Omega X}$,
\[
\dim_k\Hom_A(P_1,A)\leq\dim_k\Hom_A\bigl(A^{\dim_k\Omega X},A\bigr)
 = d\dim_k\Omega X
 \leq d^2\dim_kX.
\]
If $X$ is nonprojective, then $\Tr_AX$ is a quotient of
$\Hom_A(P_1,A)$, see \eqref{eq:tr}.  Thus
\begin{equation}\label{eq:tau-bound}
\dim_k\tau_AX
 =\dim_kD\Tr_AX
 \leq d^2\dim_kX.
\end{equation}
If $X$ is noninjective, apply the same estimate over $A^{\op}$ to $DX$.
Using \eqref{eq:inverse-translate-duality}, we obtain
\begin{equation}\label{eq:tau-inverse-bound}
\dim_k\tau_A^{-1}X\leq d^2\dim_kX.
\end{equation}

Suppose first that $X$ is noninjective.  In the almost split sequence
\[
0\longrightarrow X\longrightarrow E^+(X)
\longrightarrow\tau_A^{-1}X\longrightarrow0,
\]
every indecomposable $Y$ with $X\longrightarrow Y$ occurs as a direct
summand of $E^+(X)$. Thus the sum in \eqref{eq:local-estimate} is bounded by the dimension of the full middle term.
Using \eqref{eq:tau-inverse-bound}, we get
\[
\sum_{X\longrightarrow Y}\dim_kY
 \leq\dim_kE^+(X)
 =\dim_kX+\dim_k\tau_A^{-1}X
 \leq C_A\dim_kX.
\]
If $X$ is injective, the canonical epimorphism
$X\twoheadrightarrow X/\soc X$ is minimal left almost split, with zero
target if $X$ is simple.  Every immediate successor of $X$ is therefore a
direct summand of $X/\soc X$, and the same inequality follows.

The proof of the predecessor estimate is dual.  If $X$ is nonprojective,
use the almost split sequence
\[
0\longrightarrow\tau_AX\longrightarrow E^-(X)
\longrightarrow X\longrightarrow0
\]
and \eqref{eq:tau-bound}.  If $X$ is projective, use the minimal right almost
split inclusion $\rad X\hookrightarrow X$.  This gives the second inequality
in \eqref{eq:local-estimate}.
\end{proof}

In particular, every vertex of $\Gamma_A$ has only finitely many neighbors,
so $\Gamma_A$ is locally finite.

\subsection{Orbit lengths under additive autoequivalences}
\label{subsec:orbit-lengths}

Let $\Phi:\modcat A\longrightarrow\modcat A$ be an additive autoequivalence, not assumed to be $k$-linear, and suppose
that
\[
\dim_k\Phi(X)=\dim_kX
\quad\text{for every }X\in\modcat A.
\]
As recalled in \cref{subsubse:add-equ}, $\Phi$ preserves
irreducible morphisms and acts on the underlying graph of $\Gamma_A$.  If
the $\Phi$-orbit of an indecomposable module $X$ is finite, set
\[
o_\Phi(X)=\min\{m\geq1\mid\Phi^mX\simeq X\}.
\]

\begin{lemma}\label{lem:adjacent-orbits}
Suppose that $X\longrightarrow Y$ in $\Gamma_A$ and that both $\Phi$-orbits are finite. Set $m=o_\Phi(X)$, $n=o_\Phi(Y)$ and $g=\gcd(m,n)$. Then
\begin{equation}\label{eq:orbit-ratio}
\frac{m}{g}\frac{n}{g}\leq C_A^2.
\end{equation}
\end{lemma}

\begin{proof}
Fix an irreducible morphism $a:X\to Y$.  Since $\Phi^mX\simeq X$, for each
$j\in\mathbb Z$ choose an isomorphism
$u_j:X\xrightarrow{\sim}\Phi^{jm}X$.  Then
\[
\Phi^{jm}(a)u_j:X\longrightarrow\Phi^{jm}Y
\]
is irreducible.  The subgroup generated by $m$ in $\mathbb Z/n\mathbb Z$
has order $n/g$.  Hence the targets occurring above represent exactly
$n/g$ pairwise nonisomorphic modules, all of dimension $\dim_kY$.  The
first inequality of \cref{lem:local-estimate} gives
\begin{equation}\label{eq:first-orbit-bound}
\frac{n}{g}\dim_kY\leq C_A\dim_kX.
\end{equation}

Similarly, applying powers of $\Phi^n$ to $a$ and identifying the targets
with $Y$ produces $m/g$ pairwise nonisomorphic predecessors of $Y$, all of
dimension $\dim_kX$.  Therefore
\begin{equation}\label{eq:second-orbit-bound}
\frac{m}{g}\dim_kX\leq C_A\dim_kY.
\end{equation}
Multiplying \eqref{eq:first-orbit-bound} and
\eqref{eq:second-orbit-bound} proves \eqref{eq:orbit-ratio}.
\end{proof}

The preceding lemma also explains why adjacent vertices need not have equal
orbit length.  A power of $\Phi$ which fixes one endpoint may permute the
finitely many neighbors of that endpoint.

\begin{lemma}\label{lem:finite-orbits-propagate}
If a connected component $\mathcal C$ of $\Gamma_A$ contains one vertex
with finite $\Phi$-orbit, then every vertex of $\mathcal C$ has finite
$\Phi$-orbit.
\end{lemma}

\begin{proof}
Let $o_\Phi(X)=m$, and let $Y$ be adjacent to $X$.  Every
$\Phi^{jm}Y$ is again adjacent to $X$.  Since $X$ has only finitely many
neighbors, the $\Phi^m$-orbit of $Y$ is finite.  The full $\Phi$-orbit is
the finite union
\[
\bigcup_{r=0}^{m-1}
\Phi^r\bigl(\operatorname{Orb}_{\Phi^m}(Y)\bigr),
\]
and is therefore finite.  Induction along an unoriented path proves the
claim.
\end{proof}

For a prime $\ell$, let $v_\ell$ denote the usual $\ell$-adic valuation:
for a positive integer $n$, $v_\ell(n)$ is the largest nonnegative integer
$r$ such that $\ell^r\mid n$.
The following result is essential in the proof of \cref{thm:algebraically-closed}.
\begin{proposition}\label{prop:component-valuation}
Let $\ell>C_A^2$ be prime, and let $\mathcal C$ be a connected component of
$\Gamma_A$ containing a vertex with finite $\Phi$-orbit.  Then
\[
[X]\longmapsto v_\ell\bigl(o_\Phi(X)\bigr)
\]
is constant on $\mathcal C$.
\end{proposition}

\begin{proof}
By \cref{lem:finite-orbits-propagate}, the function is defined on every
vertex of $\mathcal C$.  Let $X$ and $Y$ be adjacent and use the notation of
\cref{lem:adjacent-orbits}.  If $v_\ell(m)\neq v_\ell(n)$, then $\ell$
divides one of the relatively prime integers $m/g$ and $n/g$.  Hence
\[
\frac{m}{g}\frac{n}{g}\geq\ell>C_A^2,
\]
contrary to \eqref{eq:orbit-ratio}.  Thus the valuation is constant across
each edge, and hence on the whole component.
\end{proof}

\section{The algebraically closed case}\label{sec:algebraically-closed}
In this section, we prove \cref{thm:main} when the base field is algebraically closed.
Throughout this section, $k$ is algebraically closed and $A$ is a
representation-infinite finite-dimensional $k$-algebra. 

\subsection{Prime-order orbits over finitely generated subfields}

The core field-theoretic tool of this section is the following lemma, where Figure~\ref{fig:field-configurations} summarizes
the two field-theoretic constructions used in the proof.

\begin{lemma}\label{lem:prime-orbits}
Let $F\subseteq k$ be finitely generated over its prime field, and let
$f(T)\in F[T]$ be nonzero.  For every sufficiently large prime $\ell$ there
exist $\lambda\in k$ and $\sigma\in\Aut_F(k)$ such that the $\sigma$-orbit
of $\lambda$ has exactly $\ell$ elements and
\[
f\bigl(\sigma^r(\lambda)\bigr)\neq0
\qquad(r\in\mathbb Z).
\]
\end{lemma}

\begin{proof}

Assume first that $\operatorname{char}k=p>0$, so that the prime field of $k$ is $\mathbb F_p$. Since $F$ is finitely generated over $\mathbb F_p$, the relative algebraic
closure of $\mathbb F_p$ in $F$ is a finite field, say
$\mathbb F_{q}$, and $F/\mathbb F_{q}$ is regular.  Hence $F$ and the algebraic extension
$\mathbb F_{q^\ell}$ are linearly disjoint over $\mathbb F_{q}$, see \cref{Linear disjointness and regular extensions}.
Their compositum
\[
E=F\mathbb F_{q^\ell}\subseteq k
\]
is therefore cyclic Galois of degree $\ell$ over $F$.  Let $\rho$ be a
generator of $\Gal(E/F)$.

If $F$ is infinite, then $E\setminus F$ is infinite: for a fixed
$u\in E\setminus F$, the elements $u+a$, $a\in F$, are pairwise distinct
and remain outside $F$.  If $F$ is finite, then $F=\mathbb F_{q}$ and $|E\setminus F|=q^\ell-q$. Since $f$ has only finitely many roots, for every sufficiently large
$\ell$ we may choose $\lambda\in E\setminus F$ with $f(\lambda)\neq0$.
The extension $E/F$ has prime degree, so the $\rho$-orbit of $\lambda$ has
length $\ell$.

Now suppose that $\operatorname{char}k=0$, so that the prime field is $\mathbb Q$.
Denote by $F_0$ the relative
algebraic closure of $\mathbb Q$ in $F$.  Then $F_0$ is a number field and
$F/F_0$ is regular.  Choose an odd prime $\ell>[F_0:\mathbb Q]$. The cyclotomic field $\mathbb Q(\zeta_{\ell^2})$ contains a cyclic
subextension $L/\mathbb Q$ of degree $\ell$, see \cref{Prime-degree cyclic extensions}.  
Put $I=L\cap F_0$.  Since
\[
\mathbb Q\subseteq I\subseteq L
\quad\text{and}\quad
[L:\mathbb Q]=\ell,
\]
the degree $[I:\mathbb Q]$ divides $\ell$ and is therefore either $1$ or
$\ell$.  On the other hand, $I\subseteq F_0$, so
\[
[I:\mathbb Q]\leq[F_0:\mathbb Q]<\ell.
\]
Consequently, $[I:\mathbb Q]=1$, and hence $L\cap F_0=\mathbb Q$.
Thus $F_0L/F_0$ is cyclic of degree $\ell$.

Since $F/F_0$ is regular,
$F$ and $F_0L$ are linearly disjoint over $F_0$, see \cref{Linear disjointness and regular extensions}.  It follows that
$E=FL\subseteq k$
is cyclic Galois of degree $\ell$ over $F$, cf. \eqref{eq:galois}.  Let $\rho$ be a generator of
$\Gal(E/F)$.  The set $E\setminus F$ is infinite, so we can choose
$\lambda\in E\setminus F$ with $f(\lambda)\neq0$.  Again the $\rho$-orbit
of $\lambda$ has length $\ell$.

In either characteristic, extend $\rho$ to an automorphism
$\sigma\in\Aut_F(k)$ by the procedure recalled in
\cref{Extension of automorphisms}.  Since $\sigma|_E=\rho$, the
$\sigma$-orbit of $\lambda$ still has length $\ell$.  Finally, the
coefficients of $f$ lie in $F$, and therefore
\[
f\bigl(\sigma^r(\lambda)\bigr)
 =\sigma^r\bigl(f(\lambda)\bigr)\neq0
\qquad(r\in\mathbb Z).
\qedhere\]
\end{proof}

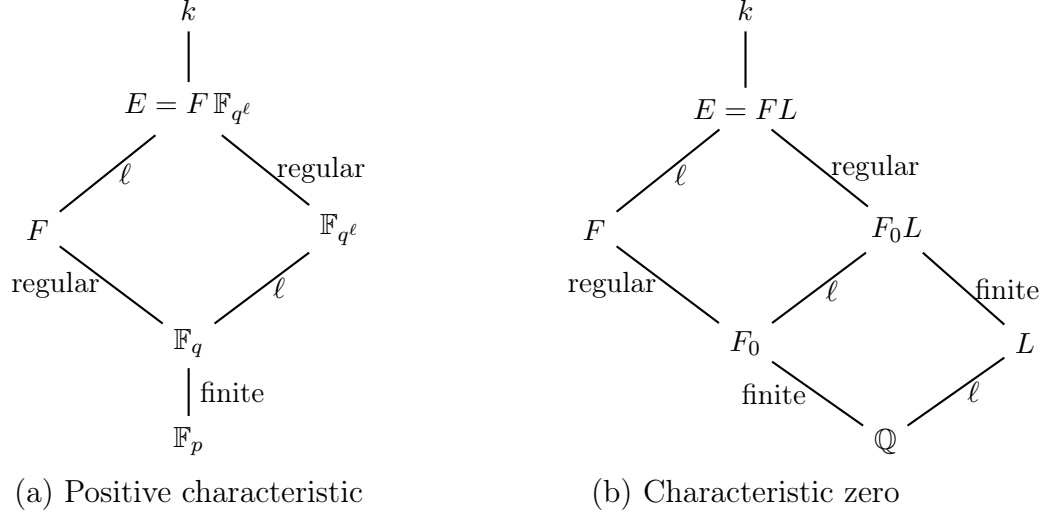
\begin{figure}[ht]
\centering

\begin{minipage}[t]{0.48\textwidth}
\centering
\begin{tikzpicture}[
  >=Latex,
  line/.style={thick},
  dashedline/.style={thick,dashed},
  every node/.style={font=\small}
]
  \node (k) at (0,4.4) {$k$};
  \node (E) at (0,3.1) {$E=F\,\mathbb{F}_{q^\ell}$};
  \node (F) at (-2.0,1.5) {$F$};
  \node (Fql) at (2.0,1.5) {$\mathbb{F}_{q^\ell}$};
  \node (Fq0) at (0,0.0) {$\mathbb{F}_{q}$};
  \node (Fp) at (0,-1.3) {$\mathbb{F}_{p}$};

  \draw[line] (k) -- (E);
  \draw[line] (E) -- node[right] {$\ell$} (F);
  \draw[line] (E) --node[right] {regular} (Fql);
  \draw[line] (F) -- node[left] {regular} (Fq0);
  \draw[line] (Fql) -- node[right] { $\ell$ } (Fq0);
  \draw[line] (Fq0) -- node[right] {finite} (Fp);


  \node[font=\normalsize] at (0,-2.0) {(a) Positive characteristic};
\end{tikzpicture}
\end{minipage}
\hfill
\begin{minipage}[t]{0.48\textwidth}
\centering
\begin{tikzpicture}[
  >=Latex,
  line/.style={thick},
  dashedline/.style={thick,dashed},
  every node/.style={font=\small}
]
  \node (k) at (0,4.4) {$k$};
  \node (E) at (0,3.1) {$E=FL$};
  \node (F) at (-2.0,1.5) {$F$};
  \node (F0L) at (2.0,1.5) {$F_0L$};
  \node (F0) at (0,0.0) {$F_0$};
  \node (L) at (3.7,0.0) {$L$};
  \node (Q) at (1.85,-1.3) {$\mathbb{Q}$};

  \draw[line] (k) -- (E);
  \draw[line] (E) -- node[right,pos=.55] {$\ell$} (F);
  \draw[line] (E) --node[right] {regular} (F0L);
  \draw[line] (F) -- node[left] {regular} (F0);
  \draw[line] (F0L) -- node[right,pos=.55] {$\ell$} (F0);
  \draw[line] (F0L) --node[right] {finite} (L);
  \draw[line] (F0) -- node[left] {finite} (Q);
  \draw[line] (L) -- node[right] {$\ell$} (Q);



  \node[font=\normalsize] at (0,-2.0) {(b) Characteristic zero};
\end{tikzpicture}
\end{minipage}

\caption{Field configurations used to construct prime-order automorphism orbits. In positive characteristic, the regularity of $F/\mathbb{F}_{q}$ implies that $F$ and $\mathbb{F}_{q^\ell}$ are linearly disjoint over $\mathbb{F}_{q}$, so $E=F\mathbb{F}_{q^\ell}$ is cyclic Galois of degree $\ell$ over $F$. 
In characteristic zero, one first chooses a cyclic extension $L/\mathbb{Q}$ of degree $\ell$ inside a cyclotomic field, observes that $L\cap F_0=\mathbb{Q}$, and then uses the regularity of $F/F_0$ to deduce that $F$ and $F_0L$ are linearly disjoint over $F_0$.}
\label{fig:field-configurations}
\end{figure}

\subsection{A field of definition and scalar twists}
\label{subsec:scalar-twists}

Fix the data provided by \cref{lem:polynomial-embedding}.  Thus, $R=k[T,h(T)^{-1}]$
for some nonzero polynomial $h(T)\in k[T]$, and ${}_AM_R$ is an
$A$--$R$-bimodule which is free of finite rank as a right $R$-module.
Moreover, the tensor functor
\[
H=M\otimes_R-:\modcat R\longrightarrow\modcat A
\]
is a representation embedding.  For each $\lambda\in k$ with
$h(\lambda)\neq0$, recall that
\[
S_\lambda=R/(T-\lambda)R
\]
is the one-dimensional $R$-module on which $T$ acts as multiplication by
$\lambda$.  The modules $S_\lambda$ are pairwise nonisomorphic, and their
images $H(S_\lambda)$ form a family of pairwise nonisomorphic
indecomposable $A$-modules of the same dimension.

\begin{lemma}\label{lem:scalar-twist}
There is a subfield $F\subseteq k$, finitely generated over the prime
field and containing the coefficients of $h$, and, for every
$\sigma\in\Aut_{F}(k)$, a dimension-preserving exact additive
autoequivalence $\Phi_\sigma:\modcat A\longrightarrow\modcat A$. These autoequivalences satisfy $\Phi_1=\id_{\modcat A}$, $\Phi_\sigma\Phi_\tau=\Phi_{\sigma\tau}$, and
\begin{equation}\label{eq:scalar-equivariance}
\Phi_\sigma\bigl(H(S_\lambda)\bigr)
 \simeq H(S_{\sigma(\lambda)})
\end{equation}
whenever $h(\lambda)\neq0$.
\end{lemma}

\begin{proof}
Choose a $k$-basis $a_1=1,a_2,\ldots,a_d$ of $A$ and a right $R$-basis
$m_1,\ldots,m_t$ of $M$.  Write
\[
a_ia_j=\sum_{u=1}^d c_{ij}^u a_u,
\qquad
a_im_s=\sum_{v=1}^t m_v f_{is}^v,
\]
where $c_{ij}^u\in k$ and $f_{is}^v\in R=k[T,h(T)^{-1}]$.  For each $f_{is}^v$ choose
$g_{is}^v(T)\in k[T]$ and $e_{is}^v\geq0$ such that
\[
f_{is}^v=\frac{g_{is}^v(T)}{h(T)^{e_{is}^v}}.
\]
Let $F$ be the subfield generated over the prime field by the
coefficients of $h$, all $c_{ij}^u$, and all coefficients of the
polynomials $g_{is}^v$.  This is a finitely generated field.

Fix $\sigma\in\Aut_{F}(k)$.  Applying $\sigma$ to coefficients and fixing
$T$ defines an automorphism $\widetilde\sigma$ of $R$.  Define
\[
\alpha_\sigma\left(\sum_{i=1}^d c_i a_i\right)
 =\sum_{i=1}^d\sigma(c_i)a_i
\qquad
\text{and}
\qquad
\theta_\sigma\left(\sum_{s=1}^t m_s f_s\right)
 =\sum_{s=1}^t m_s\widetilde\sigma(f_s).
\]
Because $\sigma$ fixes the chosen structure constants,
$\alpha_\sigma$ is a $\sigma$-semilinear ring automorphism of $A$, and
\begin{equation}\label{eq:theta-compatibility}
\theta_\sigma(amf)
 =\alpha_\sigma(a)\theta_\sigma(m)\widetilde\sigma(f)
\end{equation}
for $a\in A$, $m\in M$, and $f\in R$.

For an $A$-module $X$, let $\Phi_\sigma X$ have the same underlying
additive group as $X$, with action $a\mathbin{\cdot_\sigma}x=\alpha_{\sigma^{-1}}(a)x$. On morphisms, $\Phi_\sigma$ is the same underlying additive map.  This is an
exact additive autoequivalence with inverse $\Phi_{\sigma^{-1}}$.  Its
scalar multiplication is $c\mathbin{\cdot_\sigma}x=\sigma^{-1}(c)x$, and a $k$-basis of $X$ remains a $k$-basis after the twist, since $\sigma^{-1}$ is a field automorphism. 
So it preserves $k$-dimension.  The definitions also give the asserted
group law.

Now we prove the isomorphism \eqref{eq:scalar-equivariance}.
Define similarly an autoequivalence $\Psi_\sigma$ of $\modcat R$ by
retaining the underlying additive group and setting $f\mathbin{\cdot_\sigma}n=\widetilde{\sigma^{-1}}(f)n$. If $h(\lambda)\neq0$, then
$h(\sigma(\lambda))=\sigma(h(\lambda))\neq0$.
Under the standard identifications
\[
S_\mu=R/(T-\mu)R\simeq k
\]
induced by evaluation at $T=\mu$, the map
\[
S_{\sigma(\lambda)}\longrightarrow\Psi_\sigma(S_\lambda),
\qquad c\longmapsto\sigma^{-1}(c),
\]
is an $R$-module isomorphism.  Indeed, for $f\in R$,
\[
\sigma^{-1}\bigl(f(\sigma(\lambda))c\bigr)
 =\widetilde{\sigma^{-1}}(f)(\lambda)\sigma^{-1}(c).
\]

For $N\in\modcat R$, define
\[
\eta_N:\Phi_\sigma(M\otimes_RN)
 \longrightarrow M\otimes_R\Psi_\sigma N,
\qquad m\otimes n\longmapsto\theta_\sigma(m)\otimes n.
\]
The map is balanced: for $f\in R$,
\[
\begin{aligned}
\eta_N(mf\otimes n)
 &=\theta_\sigma(m)\widetilde\sigma(f)\otimes n\\
 &=\theta_\sigma(m)\otimes
   \bigl(\widetilde\sigma(f)\mathbin{\cdot_\sigma}n\bigr)\\
 &=\theta_\sigma(m)\otimes fn
 =\eta_N(m\otimes fn).
\end{aligned}
\]
It is $A$-linear by \eqref{eq:theta-compatibility}:
\[
\eta_N\bigl(a\mathbin{\cdot_\sigma}(m\otimes n)\bigr)
 =\theta_\sigma\bigl(\alpha_{\sigma^{-1}}(a)m\bigr)\otimes n
 =a\theta_\sigma(m)\otimes n.
\]
The same formula with $\sigma^{-1}$ gives an inverse.  The maps $\eta_N$
are natural in $N$, and hence
\[
\Phi_\sigma(M\otimes_RN)
 \simeq M\otimes_R\Psi_\sigma N.
\]
Taking $N=S_\lambda$ proves \eqref{eq:scalar-equivariance}.
\end{proof}

\subsection{Proof over an algebraically closed field}

We are now ready to prove the main theorem over an algebraically closed
field.  The proof combines the one-parameter family supplied by the
representation embedding with the semilinear twists constructed above.
By choosing parameters with successively prescribed prime orbit lengths,
\cref{prop:component-valuation} forces the resulting indecomposable modules
to lie in pairwise distinct connected components.

\begin{theorem}\label{thm:algebraically-closed}
Let $k$ be algebraically closed, and let $A$ be a representation-infinite
finite-dimensional $k$-algebra.  Then $\Gamma_A$ has infinitely many
connected components.
\end{theorem}

\begin{proof}
Let $h$, $R$, and $H$ be as in
\cref{lem:polynomial-embedding}, and let $F$ and the action
$\sigma\mapsto\Phi_\sigma$ be as in \cref{lem:scalar-twist}.
We recursively construct
parameters $\lambda_i\in k$, pairwise distinct primes $\ell_i$,
automorphisms $\sigma_i$, and indecomposable modules $X_i$.

Suppose that $\lambda_1,\ldots,\lambda_{i-1}$ have been chosen, and put $F_i=F(\lambda_1,\ldots,\lambda_{i-1})$, with $F_1=F$.  This field is finitely generated over the prime field.
Choose a prime $\ell_i$, distinct from
$\ell_1,\ldots,\ell_{i-1}$, such that $\ell_i>C_A^2$ and such that
\cref{lem:prime-orbits} applies to $F_i$, $h$, and $\ell_i$.  This is
possible because only finitely many primes have been used and the lemma
applies to every sufficiently large prime.  We obtain
\[
\lambda_i\in k,
\qquad
\sigma_i\in\Aut_{F_i}(k),
\]
such that the $\sigma_i$-orbit of $\lambda_i$ has length $\ell_i$ and $h\bigl(\sigma_i^r(\lambda_i)\bigr)\neq0$ for every $r\in\mathbb Z$. In particular, $\lambda_i\notin F_i$, so the parameters chosen at different
stages are distinct.  Set $X_i=H(S_{\lambda_i})$. The module $X_i$ is indecomposable because $H$ is a representation
embedding.

Let $\Phi_i=\Phi_{\sigma_i}$.  The group law in \cref{lem:scalar-twist} and
\eqref{eq:scalar-equivariance} give $\Phi_i^r(X_i) \simeq H(S_{\sigma_i^r(\lambda_i)})$ for every $r\in\mathbb Z$. Since $H$ reflects isomorphism classes and
$S_\mu\simeq S_\nu$ precisely when $\mu=\nu$, it follows that $o_{\Phi_i}(X_i)=\ell_i$. For $j<i$, the automorphism $\sigma_i$ fixes $\lambda_j$, since it fixes
$F_i$ pointwise.  Hence $o_{\Phi_i}(X_j)=1$.

If $X_i$ and $X_j$, with $j<i$, belonged to the same connected component,
\cref{prop:component-valuation}, applied to $\Phi_i$ and $\ell_i$, would
give $v_{\ell_i}\bigl(o_{\Phi_i}(X_i)\bigr)=v_{\ell_i}\bigl(o_{\Phi_i}(X_j)\bigr)$. The two sides are $1$ and $0$, respectively.  This contradiction shows that
$X_1,X_2,\ldots$ lie in pairwise distinct connected components of
$\Gamma_A$.
\end{proof}

\section{Separable base change}\label{sec:base-change}

Throughout this section, $k$ is perfect, $A$ is a finite-dimensional
$k$-algebra, and $K=\kkbar$.  We write
\[
A_K=K\otimes_kA,
\qquad
X_K=K\otimes_kX,
\qquad
G=\Gal(K/k),
\]
where $X\in\modcat A$.
Since $k$ is perfect, $K/k$ is a Galois extension and $K^G=k$.

A finite-dimensional semisimple algebra over a perfect field is a separable
algebra, where we refer
to \cite{DeMeyerIngraham,Pierce} for separable algebras and their behavior
under scalar extension. 
It follows that, for every finite-dimensional $k$-algebra $A$,
\begin{equation}\label{eq:jacobson-base-change-prelim}
J(K\otimes_kA)=K\otimes_kJ(A).
\end{equation}
Indeed, $K\otimes_kJ(A)$ is a nilpotent ideal of $K\otimes_kA$, and hence
$K\otimes_kJ(A)\subseteq J(K\otimes_kA)$. On the other hand,
$(K\otimes_kA)/(K\otimes_kJ(A))
 \simeq K\otimes_k(A/J(A))$
is semisimple, which gives the reverse inclusion.

Finally, if $V$ is a finite-dimensional $k$-vector space, then
\begin{equation}\label{eq:galois-invariants-tensor}
(K\otimes_kV)^G=1\otimes V,
\end{equation}
where $G$ acts on the first tensor factor.  Indeed, after expanding an
invariant tensor in a $k$-basis of $V$, each of its coefficients belongs
to $K^G=k$.

\subsection{Radicals and irreducible morphisms}
\label{subsec:radical-base-change}

We begin with a descent statement for finite-dimensional modules.
Since $K/k$ is algebraic, every finite-dimensional $A_K$-module is defined
over a finite intermediate extension.  Since $k$ is perfect, such an
extension is separable, which yields the direct-summand statement below. The lemma is proved in
\cite[Proposition~4.13]{Kasjan}; compare also
\cite[Theorem~3.3]{JensenLenzing} in the representation-finite setting. We include a short proof for the convenience of the readers.

\begin{lemma}\label{lem:descent-summand}
Every indecomposable $A_K$-module is isomorphic to a direct summand of
$X_K$ for some indecomposable $A$-module $X$.
\end{lemma}

\begin{proof}
Let $Y$ be an indecomposable $A_K$-module.  Choose a $k$-basis
$a_1,\ldots,a_d$ of $A$ and a $K$-basis of $Y$.  With respect to the
latter basis, the actions of $a_1,\ldots,a_d$ are represented by finitely
many matrices with entries in $K$.  Let $L$ be the subfield of $K$
generated over $k$ by all entries of these matrices.  The same matrices
define an $A_L$-module $Y_0$, where $A_L=L\otimes_kA$, and scalar extension
gives an isomorphism
\[
K\otimes_LY_0\simeq Y.
\]

Since $L$ is generated over $k$ by finitely many elements of the algebraic
extension $K/k$, the extension $L/k$ is finite.  Since $k$ is perfect,
$L/k$ is moreover separable.  Therefore
\[
K\otimes_kL
 \simeq
\prod_{\tau\in\Hom_k(L,K)}K.
\]
Regarding $Y_0$ as an $A$-module by restriction of scalars, we consequently
obtain
\[
K\otimes_kY_0
 \simeq
(K\otimes_kL)\otimes_LY_0
 \simeq
\bigoplus_{\tau\in\Hom_k(L,K)}
K\otimes_{L,\tau}Y_0.
\]
The summand corresponding to the inclusion
$\iota:L\hookrightarrow K$ is
$K\otimes_{L,\iota}Y_0\simeq Y$.
Thus $Y$ is a direct summand of $K\otimes_kY_0$.

Decompose the finite-dimensional $A$-module $Y_0$ into indecomposable
summands,
\[
Y_0\simeq X_1\oplus\cdots\oplus X_m.
\]
Then
\[
K\otimes_kY_0
 \simeq
(X_1)_K\oplus\cdots\oplus(X_m)_K.
\]
Since $Y$ is indecomposable, the Krull--Schmidt theorem implies that $Y$
is isomorphic to a direct summand of $(X_j)_K$ for some $j$.  This proves the assertion.
\end{proof}

\begin{lemma}\label{lem:radical-base-change}
For all $U,V\in\modcat A$, scalar extension induces an equality
\begin{equation}\label{eq:radical-base-change}
\rad_{A_K}(U_K,V_K)=K\otimes_k\rad_A(U,V).
\end{equation}
\end{lemma}

\begin{proof}
Set $E=\End_A(U\oplus V)$.  By \cite[Lemma 2.2]{Kasjan}, there is a
natural isomorphism
\[
K\otimes_kE\simeq\End_{A_K}(U_K\oplus V_K).
\]
Since $k$ is perfect, \eqref{eq:jacobson-base-change-prelim} applied to $E$ gives
\[
J(K\otimes_kE)=K\otimes_kJ(E).
\]
Let $e_U,e_V\in E$ be the idempotents associated with the summands $U$ and
$V$.  Taking the $(V,U)$-block in the equality above and using
\eqref{eq:radical-block} gives \eqref{eq:radical-base-change}.
\end{proof}

\begin{lemma}\label{lem:radical-square-base-change}
For all $U,V\in\modcat A$, scalar extension induces an equality
\begin{equation}\label{eq:radical-square-base-change}
\rad_{A_K}^2(U_K,V_K)=K\otimes_k\rad_A^2(U,V).
\end{equation}
\end{lemma}

\begin{proof}
The inclusion from right to left follows by extending factorizations.  For
the reverse inclusion, linearity reduces the problem to a single
composition
\[
U_K\xrightarrow{\alpha}N\xrightarrow{\beta}V_K
\]
with $\alpha\in\rad_{A_K}(U_K,N)$ and $\beta\in\rad_{A_K}(N,V_K)$. Decompose $N$ into indecomposable summands and apply
\cref{lem:descent-summand} to each of them.  Taking the direct sum of the
resulting $A$-modules, we obtain $W\in\modcat A$ and split maps
\[
N\xrightarrow{i}W_K\xrightarrow{p}N,
\qquad pi=1_N.
\]
Since the radical is a two-sided ideal,
\[
i\alpha\in\rad_{A_K}(U_K,W_K),
\qquad
\beta p\in\rad_{A_K}(W_K,V_K).
\]
By \cref{lem:radical-base-change}, there are radical morphisms
$f_r:U\to W$ and $g_s:W\to V$, together with scalars $a_r,b_s\in K$, such
that
\[
i\alpha=\sum_r a_r(f_r)_K,
\qquad
\beta p=\sum_s b_s(g_s)_K.
\]
Consequently,
\[
\beta\alpha=(\beta p)(i\alpha)
 =\sum_{r,s}a_rb_s(g_sf_r)_K
 \in K\otimes_k\rad_A^2(U,V).
\]
This proves the reverse inclusion.
\end{proof}

\begin{corollary}\label{cor:irr-base-change}
For all $U,V\in\modcat A$, scalar extension induces a natural isomorphism
\begin{equation*}
K\otimes_k\Irr_A(U,V)
 \simeq\Irr_{A_K}(U_K,V_K).
\end{equation*}
\end{corollary}

\begin{proof}
Tensoring over the field $k$ is exact.  Take the quotient of
\eqref{eq:radical-base-change} by
\eqref{eq:radical-square-base-change}.
\end{proof}

\subsection{Galois orbits and path lifting}\label{subsec:path-lifting}
We now study how the connected components of the Auslander--Reiten quiver
behave under scalar extension from $k$ to $K$.  The main point is that,
although an indecomposable $A$-module may become decomposable over $K$,
the Galois group acts transitively on the isomorphism classes of its
indecomposable direct summands.  Combining this transitivity with the
base-change isomorphism for irreducible morphisms allows us to lift paths
in $\Gamma_A$ to paths in $\Gamma_{A_K}$.  We will use this to show that
finiteness of the number of connected components passes from
$\Gamma_A$ to $\Gamma_{A_K}$.

For $\sigma\in G$, let $\widehat{\sigma}$ be the automorphism of $A_K$
defined by
\[
\widehat{\sigma}(\lambda\otimes a)=\sigma(\lambda)\otimes a
\qquad
(\lambda\in K,\ a\in A).
\]
For an $A_K$-module $N$, let ${}^\sigma N$ denote the twist of $N$ by
$\widehat{\sigma}$, that is, the module with the same underlying additive
group as $N$ and with $A_K$-action
\[
a\mathbin{\cdot_\sigma}n=\widehat{\sigma}^{-1}(a)n
\qquad
(a\in A_K,\ n\in N).
\]
For an $A_K$-linear map $f:N\to N'$, let
${}^\sigma f:{}^\sigma N\longrightarrow{}^\sigma N'$
be the same underlying additive map.  This defines an exact additive
autoequivalence
\[
{}^\sigma(-):\modcat A_K\longrightarrow\modcat A_K
\]
with inverse ${}^{\sigma^{-1}}(-)$.  In particular, it preserves
indecomposable modules and irreducible
morphisms.  Moreover,
${}^\sigma({}^\tau N)={}^{\sigma\tau}N$,
so these autoequivalences define an action of $G$ on the isomorphism
classes of $A_K$-modules.

If $X\in\modcat A$, define
\begin{equation}\label{eq:extended-module-twist-isomorphism}
t_{\sigma,X}:X_K\longrightarrow{}^\sigma X_K,
\qquad
\lambda\otimes x\longmapsto\sigma^{-1}(\lambda)\otimes x.
\end{equation}
This map is an $A_K$-module isomorphism.  Indeed, for
$\lambda,\mu\in K$, $a\in A$, and $x\in X$, we have
\[
\begin{aligned}
t_{\sigma,X}\bigl((\mu\otimes a)(\lambda\otimes x)\bigr)
 &=\sigma^{-1}(\mu\lambda)\otimes ax,\\
(\mu\otimes a)\mathbin{\cdot_\sigma}
 t_{\sigma,X}(\lambda\otimes x)
 &=(\sigma^{-1}(\mu)\otimes a)
   (\sigma^{-1}(\lambda)\otimes x)\\
 &=\sigma^{-1}(\mu\lambda)\otimes ax.
\end{aligned}
\]
Its inverse is given by
\[
{}^\sigma X_K\longrightarrow X_K,
\qquad
\lambda\otimes x\longmapsto\sigma(\lambda)\otimes x.
\]
Consequently, ${}^\sigma X_K\simeq X_K$ for every $\sigma\in G$, and hence
$G$ permutes the isomorphism classes of the indecomposable direct summands
of $X_K$. 
The following lemma shows that the action is transitive.

\begin{lemma}\label{lem:galois-transitivity}
Let $X$ be an indecomposable $A$-module.  The group $G$ acts transitively on
the isomorphism classes of indecomposable direct summands of $X_K$.
\end{lemma}

\begin{proof}
Set
\[
D_X=\End_A(X)/J(\End_A(X)).
\]
Since $X$ is indecomposable, $\End_A(X)$ is local and $D_X$ is a division
algebra.  The proof of \cref{lem:radical-base-change}, applied with
$U=V=X$, gives an isomorphism
\begin{equation*}
\frac{\End_{A_K}(X_K)}{J(\End_{A_K}(X_K))}
 \simeq K\otimes_kD_X.
\end{equation*}
This isomorphism is $G$-equivariant, where $G$ acts on the right-hand side
through the first tensor factor.  To see this explicitly, let
$t_{\sigma,X}:X_K\to{}^\sigma X_K$ be the isomorphism in
\eqref{eq:extended-module-twist-isomorphism}.  The action on
$\End_{A_K}(X_K)$ is $f\longmapsto t_{\sigma,X}^{-1}\,{}^\sigma f\,t_{\sigma,X}$. Under the standard identification
$K\otimes_k\End_A(X)\simeq\End_{A_K}(X_K)$, this sends
$\lambda\otimes u$ to $\sigma(\lambda)\otimes u$.

Write
\[
X_K\simeq\bigoplus_{j=1}^sY_j^{n_j},
\]
where the $Y_j$ are pairwise nonisomorphic indecomposable modules.
Krull--Schmidt theory gives
\[
\frac{\End_{A_K}(X_K)}{J(\End_{A_K}(X_K))}
 \simeq
 \prod_{j=1}^s
 M_{n_j}\left(
 \frac{\End_{A_K}(Y_j)}{J(\End_{A_K}(Y_j))}
 \right).
\]
The central primitive idempotents of this semisimple algebra are therefore
indexed by the isotypic summands $Y_j^{n_j}$, and hence by the isomorphism
classes $[Y_j]$.  Under the $G$-action, the idempotent corresponding to
$Y_j^{n_j}$ is sent to the idempotent corresponding to the isotypic summand
of ${}^\sigma Y_j$.  Thus the action on central primitive idempotents is the
same as the action on the isomorphism classes of indecomposable summands.

Note that the set of central primitive idempotents is finite. If the action had more than one orbit, choose one such orbit and let $e$
be the sum of the central primitive idempotents belonging to it.  Then
$e$ would be a nonzero $G$-invariant idempotent different from $1$.
Applying \eqref{eq:galois-invariants-tensor} to the underlying finite-dimensional $k$-vector space of $D_X$, we obtain
\[
(K\otimes_kD_X)^G=1\otimes D_X.
\]
Hence $e=1\otimes d$ for an idempotent $d\in D_X$.  A division algebra has
no idempotents other than $0$ and $1$, contradicting the choice of $e$.
The action is therefore transitive.
\end{proof}

We now use the preceding transitivity result to lift paths in the
Auslander--Reiten quiver of $A$.  More precisely, once an indecomposable
summand over the initial vertex is prescribed, the remaining vertices of
the path can be lifted successively.

\begin{lemma}\label{lem:path-lifting}
Let $[X_0]-[X_1]-\cdots-[X_m]$ be an unoriented path in $\Gamma_A$, and let $\widetilde X_0$ be an indecomposable direct summand of $(X_0)_K$. There are indecomposable direct summands $\widetilde X_i\mid(X_i)_K$ $(1\leq i\leq m)$ such that $[\widetilde X_0]-[\widetilde X_1]-\cdots-[\widetilde X_m]$ is an unoriented path in $\Gamma_{A_K}$.
\end{lemma}

\begin{proof}
It is enough to lift one edge while prescribing the summand over its initial
vertex.  Suppose first that $X_i\longrightarrow X_{i+1}$.  By
\cref{cor:irr-base-change},
\[
\Irr_{A_K}((X_i)_K,(X_{i+1})_K)\neq0.
\]
Choose decompositions into indecomposable modules,
\[
(X_i)_K\simeq\bigoplus_aU_a,
\qquad
(X_{i+1})_K\simeq\bigoplus_bV_b.
\]
The radical and its square are additive in each variable, so
\[
\Irr_{A_K}((X_i)_K,(X_{i+1})_K)
 \simeq\bigoplus_{a,b}\Irr_{A_K}(U_a,V_b).
\]
Hence there are summands $U\mid(X_i)_K$ and
$V\mid(X_{i+1})_K$ and an irreducible morphism $f:U\to V$.

By \cref{lem:galois-transitivity}, there is $\sigma\in G$ and an
isomorphism $\varphi:{}^\sigma U\xrightarrow{\sim}\widetilde X_i$. The morphism ${}^\sigma f\,\varphi^{-1}: \widetilde X_i\longrightarrow{}^\sigma V$ is irreducible.  
Transporting the summand ${}^\sigma V$ along a fixed isomorphism \({}^\sigma(X_{i+1})_K\xrightarrow{\sim}(X_{i+1})_K\), we obtain an indecomposable direct summand
$\widetilde X_{i+1}\mid(X_{i+1})_K$ isomorphic to ${}^\sigma V$.

If the edge is oriented as $X_{i+1}\longrightarrow X_i$, the same argument is applied with the prescribed summand as target.  Starting with an arrow $f:U\to V$, where $V\mid(X_i)_K$, choose $\psi:{}^\sigma V\xrightarrow{\sim}\widetilde X_i$ and use the irreducible morphism $\psi\,{}^\sigma f:{}^\sigma U\longrightarrow\widetilde X_i$. Induction on $i$ completes the lift.
\end{proof}

The following is the main result of this subsection.
\begin{proposition}\label{prop:components-base-change}
If $\Gamma_A$ has only finitely many connected components, then
$\Gamma_{A_K}$ has only finitely many connected components.
\end{proposition}

\begin{proof}
Choose indecomposable modules $M_1,\ldots,M_t$, one from each connected
component of $\Gamma_A$.  Each $(M_i)_K$ has only finitely many
indecomposable direct summands.  Let $\mathcal S$ be the finite set of their
isomorphism classes.

Let $N$ be an indecomposable $A_K$-module.  By
\cref{lem:descent-summand}, there is an indecomposable $A$-module $X$ such
that $N\mid X_K$.  The vertex $[X]$ is connected in $\Gamma_A$ to some
$[M_i]$.  Apply \cref{lem:path-lifting} to a path from $[X]$ to $[M_i]$,
with $N$ as the prescribed initial summand.  The lifted path connects $[N]$
to a member of $\mathcal S$.  Thus every connected component of
$\Gamma_{A_K}$ meets the finite set $\mathcal S$, and there are only
finitely many such components.
\end{proof}

\subsection{Proof over a perfect field}

\begin{proof}[Proof of \cref{thm:main}]
Let $K=\kkbar$.  Since $k$ is perfect, the algebraic extension $K/k$ is
separable and hence MacLane separable.  By Jensen and Lenzing's base-field
theorem, see \cref{Perfect fields and Galois extensions}, $A$ is of finite representation
type if and only if $A_K$ is.  Thus $A_K$ is representation-infinite.

The field $K$ is algebraically closed, so
\cref{thm:algebraically-closed} shows that $\Gamma_{A_K}$ has infinitely
many connected components.  If $\Gamma_A$ had only finitely many connected
components, \cref{prop:components-base-change} would imply that
$\Gamma_{A_K}$ also had only finitely many.  This contradiction proves the
theorem.
\end{proof}

\section*{Acknowledgements}
The authors thank Jes\'us Efr\'en P\'erez Terrazas for helpful comments and for bringing the perfect-field version of the tame case, implicit in~\cite[Theorem~3.2 and Corollary~6.16]{CanDeVicentePerez}, to their attention.
The second author thanks Michael Shapiro for teaching a seminar-style course at
Michigan State University from January to May 2025, based on Schiffler's
book~\cite{Schiffler}, through which he learned the foundations of the
representation theory of algebras.

\section*{Statement on AI usage}

The main theorem and the overall proof strategy originated from the human authors. In particular, the first author realized that the representation embeddings developed in \cite{Bongartz,BPS} can be employed to distinguish the distinct components of the AR-quiver of representation-infinite algebras. The second author subsequently proposed that the orbit lengths of adjacent vertices in the AR-quiver under category self-equivalences can be used, via large primes, to define an invariant of the components of the AR-quiver. 

ChatGPT assisted
with the technical construction and verification of the semilinear
twist argument in Lemma~\ref{lem:scalar-twist}.
ChatGPT was used in extending the result from
algebraically closed fields to arbitrary perfect fields.  In
particular, it assisted with the formulation and technical details of
Lemmas~\ref{lem:radical-square-base-change},
\ref{lem:galois-transitivity}, and \ref{lem:path-lifting}, and
Proposition~\ref{prop:components-base-change}.  The authors subsequently
checked and revised all AI-assisted arguments and take full
responsibility for the contents of the paper.


\begin{thebibliography}{99}
\bibitem[AR75]{AR75}
M.~Auslander and I.~Reiten,
Representation theory of Artin algebras III: Almost split sequences,
\emph{Comm. Algebra} \textbf{3} (1975), no.~3, 239--294.

\bibitem[ARS89]{ARS-Galois}
M.~Auslander, I.~Reiten, and S.~O.~Smal{\o},
Galois actions on rings and finite Galois coverings,
\emph{Math. Scand.} \textbf{65} (1989), no.~1, 5--32.

\bibitem[ARS95]{ARS}
M.~Auslander, I.~Reiten, and S.~O.~Smal{\o},
\emph{Representation Theory of Artin Algebras},
Cambridge Studies in Advanced Mathematics, vol.~36,
Cambridge University Press, Cambridge, 1995.

\bibitem[Ba85]{Bautista85}
R.~Bautista,
On algebras of strongly unbounded representation type,
\emph{Comment. Math. Helv.} \textbf{60} (1985), no.~3, 392--399.

\bibitem[BPS25]{BPS}
R.~Bautista, E.~P\'erez, and L.~Salmer\'on,
A representation embedding for algebras of infinite type,
\emph{J. Pure Appl. Algebra} \textbf{229} (2025), no.~6,
Paper No.~107955.

\bibitem[BCB24]{BennettCrawleyBoevey}
R.~Bennett-Tennenhaus and W.~Crawley-Boevey,
Semilinear clannish algebras,
\emph{Proc. Lond. Math. Soc.} (3) \textbf{129} (2024), no.~4,
Paper No.~e12637.

\bibitem[B85]{Bongartz85}
K.~Bongartz,
Indecomposables are standard,
\emph{Comment. Math. Helv.} \textbf{60} (1985), no.~3, 400--410.

\bibitem[B16]{Bongartz}
K.~Bongartz,
Representation embeddings and the second Brauer--Thrall conjecture,
preprint, arXiv:1611.02017v5 [math.RT], 2023.

\bibitem[BHK25]{BHK25}
E.~D.~B{\o}rve, E.~J.~Hanson, and M.~Kaipel,
Bricks and $\tau$-tilting theory under base field extensions,
preprint, arXiv:2508.01040 [math.RT], 2025.

\bibitem[CDP14]{CanDeVicentePerez}
\'A.~F.~Can Cabrera, J.~P.~De-Vicente Moo, and J.~E.~P\'erez Terrazas,
Differential tensor algebras and endolength vectors,
\emph{Colloq. Math.} \textbf{136} (2014), no.~1, 75--98.

\bibitem[CB88]{CrawleyBoevey}
W.~W.~Crawley-Boevey,
On tame algebras and bocses,
\emph{Proc. Lond. Math. Soc.} (3) \textbf{56} (1988), no.~3, 451--483.

\bibitem[DI71]{DeMeyerIngraham}
F.~DeMeyer and E.~Ingraham,
\emph{Separable Algebras over Commutative Rings},
Lecture Notes in Mathematics, vol.~181,
Springer-Verlag, Berlin--New York, 1971.

\bibitem[DR78]{DlabRingel}
V.~Dlab and C.~M.~Ringel,
The representations of tame hereditary algebras,
in \emph{Representation Theory of Algebras} (Proc. Conf., Temple Univ.,
Philadelphia, 1976), Lecture Notes in Pure and Appl. Math., vol.~37,
Dekker, New York, 1978, pp.~329--353.

\bibitem[F00]{Farnsteiner00}
R.~Farnsteiner,
On the Auslander--Reiten quiver of an infinitesimal group,
\emph{Nagoya Math. J.} \textbf{160} (2000), 103--121.


\bibitem[JL82]{JensenLenzing}
C.~U.~Jensen and H.~Lenzing,
Homological dimension and representation type of algebras under base field
extension,
\emph{Manuscripta Math.} \textbf{39} (1982), no.~1, 1--13.

\bibitem[K00]{Kasjan}
S.~Kasjan,
Auslander--Reiten sequences under base field extension,
\emph{Proc. Amer. Math. Soc.} \textbf{128} (2000), no.~10, 2885--2896.

\bibitem[K13]{Kuelshammer13}
J.~K\"ulshammer,
Representation type of Frobenius--Lusztig kernels,
\emph{Quart. J. Math.} \textbf{64} (2013), no.~2, 471--488.

\bibitem[K15]{KuelshammerCorrigendum}
J.~K\"ulshammer,
Corrigendum ``Representation type of Frobenius--Lusztig kernels'',
\emph{Quart. J. Math.} \textbf{66} (2015), no.~4, 1139.

\bibitem[L02]{Lang}
S.~Lang,
\emph{Algebra}, 3rd ed.,
Graduate Texts in Mathematics, vol.~211,
Springer, New York, 2002.

\bibitem[MS24]{MalicSchroll}
G.~Mali\'c and S.~Schroll,
Dessins d'enfants, Brauer graph algebras and Galois invariants,
\emph{Algebr. Represent. Theory} \textbf{27} (2024), 655--665.

\bibitem[P82]{Pierce}
R.~S.~Pierce,
\emph{Associative Algebras},
Graduate Texts in Mathematics, vol.~88,
Springer-Verlag, New York--Berlin, 1982.

\bibitem[R78]{Ringel78}
C.~M.~Ringel,
Finite dimensional hereditary algebras of wild representation type,
\emph{Math. Z.} \textbf{161} (1978), no.~3, 235--255.

\bibitem[R80]{Ringel}
C.~M.~Ringel,
Report on the Brauer--Thrall conjectures: Rojter's theorem and the theorem
of Nazarova and Rojter (on algorithms for solving vectorspace problems. I),
in \emph{Representation Theory I} (Proc. Workshop, Carleton Univ., Ottawa,
1979), Lecture Notes in Mathematics, vol.~831,
Springer, Berlin, 1980, pp.~104--136.

\bibitem[R68]{Roiter}
A.~V.~Roiter,
Unbounded dimensionality of indecomposable representations of an algebra
with an infinite number of indecomposable representations,
\emph{Math. USSR-Izv.} \textbf{2} (1968), no.~6, 1223--1230.

\bibitem[S14]{Schiffler}
R.~Schiffler,
\emph{Quiver Representations},
CMS Books in Mathematics,
Springer, Cham, 2014.


\bibitem[W97]{Washington}
L.~C.~Washington,
\emph{Introduction to Cyclotomic Fields}, 2nd ed.,
Graduate Texts in Mathematics, vol.~83,
Springer, New York, 1997.


\end{thebibliography}
\end{document}